\documentclass[a4paper,11pt]{amsart}
\usepackage[pdftex]{color,graphicx}
\usepackage{amssymb}
\usepackage{amsfonts}
\usepackage{esint}
\usepackage{amsmath}
\usepackage{amsrefs}
\usepackage{epsfig}
\usepackage{amsthm}
\usepackage{color}
\usepackage[]{epsfig}
\usepackage[]{pstricks}
\usepackage[utf8]{inputenc}

\newtheorem{theorem}{Theorem}[section]

\newtheorem{lemma}{Lemma}[section]
\newtheorem{corollary}{Corollary}[section]

\theoremstyle{definition}

\theoremstyle{remark}

\newcommand{\ran}{\rangle}
\newcommand{\lan}{\langle}

\newcommand{\interno}[2]{\left\langle #1 ,#2 \right\rangle}

\DeclareMathOperator{\trace}{trace}

\numberwithin{equation}{section}
\title[Ruled translating solitons of IMCF in $\mathbb{L}^3$]{Ruled surfaces as translating solitons of the inverse mean curvature flow in the three-dimensional Lorentz-Minkowski space}
\author{Greg\'orio Silva Neto \and Vanessa Silva}
\date{April 26, 2023}

\address{Instituto de Matem\'atica,
Universidade Federal de Alagoas,
Macei\'o, AL, 57072-900, Brazil}
\email{gregorio@im.ufal.br}

\address{Instituto Federal de Alagoas, Maceió, AL, 57020-510, Brazil}
\email{vanessa.silva@im.ufal.br}

\begin{document}
\footnotetext{G. Silva Neto was partially supported by National Council for Scientific and Technological Development and the Alagoas Research Foundation of Brazil}
\subjclass[2020]{53E10, 53C50, 53C42}
\begin{abstract}
In this paper, we classify the nondegenerate ruled surfaces in the three-dimensional Lorentz-Minkowski space that are translating solitons for the inverse mean curvature flow. In particular, we prove the existence of non-cylindrical ruled translating solitons, which contrast with the Euclidean setting.
\end{abstract}
\maketitle

\section{Introduction}
Let $X:M\rightarrow\mathbb{L}^3$ be a smooth immersion of a nondegenerate surface $M,$ with nonzero mean curvature, in the three-dimensional Lorentz-Minkowski space $\mathbb{L}^3.$ The inverse mean curvature flow (IMCF) in the Lorentz-Minkowski space is a one parameter family of immersions satisfying
$X_t=X(.,t):M\rightarrow \mathbb{L}^3$ with $M_t=X_t(M),$ satisfying
\begin{equation}\label{def-fcmi}
\left\{\begin{aligned}
 \frac{\partial}{\partial t}X(p,t)=-\frac{1}{H(p,t)}N(p,t)\\
X(p,0)=X(p),\hspace{0.5cm} p\in M,
\end{aligned}\right.
\end{equation}
where $H(\cdot,t)$ and $N(\cdot,t)$ are the mean curvature and the unitary normal vector field of $M_t,$ respectively, and $t\in[0,T)$.

There exists a considerable literature about the IMCF in the Euclidean space, and this flow has been shown to be very useful to obtain geometric inequalities, such as the Penrose inequality for asymptotically Euclidean Riemannian manifolds, proven by Huisken and Ilmanen in \cite{luc2}. In their turn, Gerhardt and Urbas proved in \cite{luc1} and \cite{luc3} the existence of a unique solution of the flow \eqref{def-fcmi} for every $t\geq0$ in the case when the initial surface is starshaped. They proved that, when $t\to\infty,$ this flow converges to a round sphere after rescaling.

In the subject of geometric inequalities, Guan and Li, see \cite{luc4}, used the IMCF to prove Alexandrov-Fenchel type inequalities. In this topic, we can also cite the works \cite{luc6}, \cite{luc7} and \cite{luc8}.
In the non-Euclidean setting, Gehardt proved in \cite{luc9} some existence and regularity results for the mean curvature flow in cosmological spacetimes. In \cite{luc10} and \cite{luc11} we can find results of existence and regularity for the mean curvature flow in asymptotically flat pseudo-Euclidean manifolds. The existence and regularity of the IMCF in pseudo-Riemannian manifolds, with emphasis in Lorentzian manifolds, are discussed by Gerhardt in \cite{luc12}. The IMCF on Lorentz-Minkowski spaces appears in the classic problem of finding a smooth transition from a big crunch spacetime to a spacelike big bang, under the Einstein equations and the state equation for perfect fluids. Also in \cite{luc12}, Gerhardt proved, after redimensioning, that the IMCF in semi-Riemannian spaces could be used to define this transition. He also proved the existence of the IMCF for every time in cosmological spacetimes.


A smooth surface $X_0:M\to\mathbb{L}^3$ is called a translating soliton (or translator) of the IMCF, relative to the nonzero velocity vector $V\in\mathbb{L}^3,$ if the one parameter family of immersions 
\begin{equation}\label{eq-soliton}
    X(\cdot,t)=X_0+\phi(t)V
\end{equation}
is a solution of \eqref{def-fcmi}, for some smooth and positive function $\phi(t)$ such that $\phi(0)=0$. Geometrically, this means that the surface $X_0(M)$ moves under the flow only by translations in the direction $V.$ As we will prove later, Equation \eqref{eq-soliton} will be equivalent to the surface $M$ satisfy
\begin{equation}\label{princ-soliton-1}
    \interno{N}{V}=-\frac{1}{H},
\end{equation}
that will be called the equation of the translating solitons.

In its turn, a ruled surface $M\subset \mathbb{L}^3$ is the union of a one parameter family of straight lines. Since it is an affine concept, the definition of ruled surfaces in $\mathbb{L}^3$ remains the same as in $\mathbb{R}^3.$ Thus, a ruled surface can be parametrized locally by
\begin{equation}\label{eq-def-ruled}
    X(s,t)=\gamma(s)+t\beta(s),
\end{equation}
where $t\in\mathbb{R}$, $\gamma:I\subset\mathbb{R}\rightarrow\mathbb{L}^3$ is a smooth curve in $\mathbb{L}^3$ and $\beta(s)$ is a vector field along $\gamma$. The curve $\gamma(s)$ is called a directrix (or the base curve) of the surface $M$ and the lines $\gamma(s)+t\beta(s),$ for fixed $s,$ are called the generators of the surface. If $\gamma$ reduces to a point, the surface is called conical. On the other hand, if the lines are all parallel to a fixed direction, i.e., $\beta(s)$ is a constant vector, then the surface is called cylindrical.

We recall that Hieu and Hoang proved in \cite{luc14} that ruled surfaces that are translating solitons for the mean curvature flow in $\mathbb{R}^3$ are parts of planes or cylinders over grim reaper curves. In this direction, Aydin and L\'opez classified in \cite{luc15} all the ruled surfaces that are translating solitons for the mean curvature flow in $\mathbb{L}^3$ obtaining a much richer classification than that of the Euclidean space. In particular, they obtained non-cylindrical ruled surfaces as translating solitons. On the other hand, Kim and Pyo, see \cite{luc13} classified all the ruled surfaces that are translating solitons of the IMCF in $\mathbb{R}^3,$ proving that they are cylinders over cycloids. We can also cite the work of Aydin and L\'opez, see \cite{AL}, about translating solitons in $\mathbb{R}^3$ flowing by the powers of the Gaussian curvature, including ruled surfaces, and the paper of L\'opez, see \cite{L}, about ruled surfaces that are generalized self-similar solutions of the mean curvature flow.

In this paper, we classify all the nondegenerate ruled surfaces that are translating solitons of the inverse mean curvature flow in the three-dimensional Lorentz-Minkowski space $\mathbb{L}^3.$ The paper is organized as follows: In the section \ref{prel}, for the sake of completeness, we recall some basic concepts and results of differential geometry and linear algebra in $\mathbb{L}^3,$ as well as we derive the basic equations that will be useful in the proofs. In the section \ref{noncyl} we classify the non-cylindrical ruled surfaces that are translating solitons and, in the section \ref{cyl}, we conclude the paper by classifying all the cylindrical ruled surfaces that are translating solitons of the IMCF.

\section{Preliminaries}\label{prel}

The three-dimensional Lorentz-Minkowski space, denoted by $\mathbb{L}^3$, is the vector space $\mathbb{R}^3$ with the usual vector operations, endowed with the Lorentzian metric
\[
\interno{x}{y}=x_1y_1+x_2y_2-x_3y_3,
\]
where $x=(x_1,x_2,x_3)$ and $y=(y_1,y_2,y_3)$. This metric is a indefinite bilinear symmetric form and gives rise to the norm 
\[
|v|=\sqrt{|\langle v,v\rangle|}.
\]

The vectors in $\mathbb{L}^3$ are classified by their causal type. We say that a vector $v\in\mathbb{L}^3$ is spacelike, if $\interno{v}{v}>0$, or $v=0$; timelike, if $\interno{v}{v}<0$; and lightlike, if $\interno{v}{v}=0$, but $v\neq0$. In its turn, remember that a nontrivial linear subspace $\mathcal{U}\subseteq\mathbb{L}^3$ is
\begin{enumerate}
\item[(i)] spacelike, if $\interno{.}{.}|_{\mathcal{U}}$ is positive definite;
\item[(ii)] timelike, if $\interno{.}{.}|_{\mathcal{U}}$ is negative-definite, or indefinite and nondegenerate;
\item[(iii)] lightlike, if $\interno{.}{.}|_{\mathcal{U}}$  is degenerate.
\end{enumerate}

The concepts of orthogonality and orthonormality of vectors are the same as in $\mathbb{R}^3.$ In the following, we list some basic results, relative to the orthogonality of vectors in $\mathbb{L}^3,$ that will be useful to prove our results.

\begin{lemma}\label{basic-1}
The vectors of $\mathbb{L}^3$ has the following basic properties:
\begin{itemize}
    \item[(i)] Two lightlike vectors of $\mathbb{L}^3$ are orthogonal if and only if they are linearly dependent;
    \item[(ii)] There are not two timelike orthogonal vectors in $\mathbb{L}^3;$
    \item[(iii)] There are no timelike vectors in $\mathbb{L}^3$ orthogonal to lightlike vectors.
\end{itemize}
\end{lemma}
 %
Let $u=(u_1,u_2,u_3)$, $v=(v_1,v_2,v_3)\in\mathbb{L}^3$. The vector product of $u$ and $v$ is the only vector $u\times v\in\mathbb{\mathbb{L}}^3$ such that
\[
\interno{u\times v}{x}:=\det(x,u,v)\text{, for every }x\in\mathbb{L}^3.
\]
We also define the mixed product of the vectors $u,v,w\in\mathbb{L}^3$ by
\begin{equation}\label{mixed}
(u,v,w):=\interno{u\times v}{w}=\det(u,v,w).
\end{equation}
More information about the differential geometry of surfaces in $\mathbb{L}^3$ can be found, for example, in \cite{lopez-survey} and \cite{book-CL}.

A smooth parametrized curve $\alpha:I\subset\mathbb{R}\rightarrow\mathbb{L}^3$, where $I$ is an interval, is regular if $\alpha'(t)\neq0$, for every $t\in I$. If there exists $t_0\in I$ such that $\alpha'(t_0)=0$, we say that $t_0 $  is a singular point of $\alpha.$ 

If $\alpha:I\subset\mathbb{R}\rightarrow\mathbb{L}^3$ is a regular curve, then we say that $\alpha$ is spacelike, timelike, or lightlike at $t_0\in I,$ if $\alpha'(t_0)$ is spacelike, timelike, or lightlike, respectively. Each of these situations is called the causal type of the curve. If the curve has the same causal type for every $t\in I,$ we say that this is the causal type of the curve $\alpha.$ We remark that, since $f(t)=\interno{\alpha'(t)}{\alpha'(t)}$  is a continuous function, the property of being spacelike or timelike is an open property, that is, if it holds for a point, it holds for an interval around this point. The same does not happen for the lightlike property. However, since our classification will be essentially local (we will provide the parametrizations of the surfaces), in this article we will deal only with curves with the same causal type for all of their points. 

Let $M$ be a surface in $\mathbb{L}^3.$ If $M$ is parametrized in local coordinates by $X:U\subset\mathbb{R}^2\to\mathbb{L}^3,$ where $X=X(s,t),$ the first fundamental form of $M$ has the coefficients 
\[
E=\lan X_s,X_s\ran, \quad F=\lan X_s,X_t\ran \quad \mbox{and} \quad G=\lan X_t,X_t\ran.
\]
This surface is smooth, or nondegenerate, if the map $X$ has derivatives of all orders and $|X_s\times X_t|=\sqrt{|EG-F^2|}\neq 0$ at all of its points.

A regular surface $M$ of $\mathbb{L}^3$ is spacelike, timelike, or lightlike, if its tangent plane $T_pM$ is spacelike, timelike, or lightlike, respectively, for every $p\in M.$ In $X(U)\subset M$ it is defined two unit normal vector fields
\begin{equation}\label{Normal}
N=\pm\frac{X_s\times X_t}{|X_s\times X_t|}.    
\end{equation}
If $M$ is orientable, then we can define, globally, a unit normal vector field $N:M\to\mathbb{L}^3,$ called the Gauss map of $M$, whose expression in local coordinates is given by one of the expressions in \eqref{Normal}. A choice of one of these unit normal vector fields is called a choice of orientation, and, after this choice, the surface is said to be oriented. This normal vector field satisfies 
\[
\lan N,N\ran=\epsilon,
\]
where $\epsilon=-1$ if $M$ is spacelike and $\epsilon=1$ if $M$ is timelike. The second fundamental form of $M$ in $\mathbb{L}^3$ is the bilinear and symmetric form $II_p:T_pM\times T_pM\to (T_pM)^\perp$ defined by
\begin{equation}
    \lan II_p(v,w),N(p)\ran=\lan -dN_p(v),w\ran,
\end{equation}
where $dN_p:T_pM\to T_pM$ is the differential of $N$ at $p.$ The second fundamental form can be expressed, in local coordinates, as
\[
\begin{aligned}
e&=\langle N,X_{ss}\rangle=\frac{(X_s,X_t,X_{ss})}{|X_s\times X_t|},\\
f&=\langle N,X_{st}\rangle=\frac{(X_s,X_t,X_{st})}{|X_s\times X_t|},\\
g&=\langle N,X_{ss}\rangle=\frac{(X_s,X_t,X_{tt})}{|X_s\times X_t|},\\
\end{aligned}
\]
where $(u,v,w)$ is the mixed product given by \eqref{mixed}.
If $M\subset\mathbb{L}^3$ is a nondegenerate surface and $N$ is its Gauss map, then the mean curvature vector of $M$ at $p$ is defined by
$$\textbf{H}(p)=\frac{1}{2}\trace(II_p).$$
Moreover, the mean curvature of $M$ at $p$ is the scalar $H(p)$ given by
$$\textbf{H}(p)=H(p)N(p).$$
Notice that this implies 
$$H=\epsilon\interno{\textbf{H}}{N}.$$
In local coordinates, the mean curvature has the expression
 \begin{equation}\label{H-local}
 \begin{aligned}    
H&=\dfrac{\epsilon}{2}\dfrac{eG-2fF+Eg}{EG-F^2}\\
&=-\dfrac{1}{2}\dfrac{G(X_s,X_t,X_{ss})-2F(X_s,X_t,X_{st})+E(X_s,X_t,X_{tt})}{\vert EG-F^2\vert^{3/2}},
\end{aligned}
  \end{equation}
where we used that $\vert X_s\times X_t\vert^2=-\epsilon (EG-F^2)$.

If $X(\cdot,t)$ is a one parameter family of immersions satisfying \eqref{eq-soliton}, then replacing this equation in \eqref{def-fcmi} and noticing that translations do not change the mean curvature nor the normal vector field, we obtain
\begin{equation}\label{soliton}
    \phi'(t)V=-\frac{1}{H(p,0)}N(p,0).
\end{equation}
Write, for simplicity, $H(p,0)=H(p)$ and $N(p,0)=N(p).$ Taking the inner product of \eqref{soliton} with $N(p)$ and observing that, for nondegenerate surfaces in $\mathbb{L}^3,$ it holds $\langle N(p),N(p)\rangle=\epsilon\in\{-1,1\},$ we obtain 
\begin{equation}\label{eq-sol}
    \phi'(t)\interno{N(p)}{V}=-\frac{\epsilon}{H(p)}.
\end{equation}
Since only $\phi'(t)$ depends on $t$ in \eqref{eq-sol}, this gives that $\phi'(t)$ is constant. Thus, by changing the size and the direction of $V$ if necessary (since $V$ is constant and arbitrary), we can take $\phi(t)=t$ in \eqref{eq-sol} and we can ignore $\epsilon,$ obtaining Equation \eqref{princ-soliton-1}.

Now, consider a ruled surface in $\mathbb{L}^3$ parametrized by \eqref{eq-def-ruled}. With this parametrization, we have
\[
X_t=\beta,\ X_s=\gamma'+t\beta',\ X_{ss}=\gamma''+t\beta'',\ X_{st}=\beta'\ \mbox{and}\ X_{tt}=0.
\] 
Since $X_{tt}=0$ we can rewrite the mean curvature equation \eqref{H-local} as
$$
H=-\frac{1}{2}\displaystyle\frac{G(X_s,X_t,X_{ss})-2F(X_s,X_t,X_{st})}{\vert EG-F^2\vert^{3/2}}.
$$
Replacing \eqref{H-local} in \eqref{princ-soliton-1}, we obtain
$$\frac{1}{\sqrt{\vert EG-F^2\vert}}(X_s,X_t,V)\left[\frac{G(X_s,X_t,X_{ss})-2F(X_s,X_t,X_{st})}{(EG-F^2)^{3/2}}\right]=2$$
or, equivalently,
\begin{equation}\label{princ}
(X_s,X_t,V)[G(X_s,X_t,X_{ss})-2F(X_s,X_t,X_{st})]=2(EG-F^2)^2.
\end{equation} 
Equation \eqref{princ} will be the translating soliton equation we will study in this paper.

\section{Non-cylindrical translating solitons}\label{noncyl}

We will start analyzing the case where the ruled surface is not cylindrical, i.e., the surface has the parametrization $X(s,t)=\gamma(s)+t\beta(s)$ where $\gamma:I\subset\mathbb{R}\longrightarrow\mathbb{L}^3$ is a regular curve and $\beta'(s)\neq0$ in the interval $I$. 
\begin{lemma}
Let $M$ be a non-cylindrical smooth ruled surface $\mathbb{L}^3,$ parametrized by $X(s,t)=\gamma(s)+t\beta(s)$. If $M$ is a translating soliton of the inverse mean curvature flow, relative to the velocity vector $V=(v_1,v_2,v_3)\in\mathbb{L}^3$, then $\beta(s)$ is not a lightlike vector field.
\end{lemma}
\begin{proof}
The coefficients of the first fundamental form of $M$ are given by
 \[
\begin{cases}
E&=\interno{\gamma'}{\gamma'}+2t\interno{\gamma'}{\beta'}+\interno{\beta'}{\beta'}t^2,\\
F&=\langle \gamma'+t\beta',\beta\rangle,\\
G&=\interno{\beta}{\beta}.\\
\end{cases}
 \]
Let us show that $\beta(s)$ cannot be a lightlike vector field. In fact, otherwise, we have $G=\interno{\beta}{\beta}=0$ and thus $\interno{\beta}{\beta'}=0$, which implies that $F=\interno{\gamma'}{\beta}\neq 0$, since the surface is nondegenerate. It follows, from \eqref{princ}, that
 $$
 (X_s,X_t,V)[-2F(X_s,X_t,X_{st})]=2(-F^2)^2,
 $$	
 i.e.,
\begin{equation}\label{II}
(X_s,X_t,V)(X_s,X_t,X_{st})=- F^3.
\end{equation}	
Since
$$
X_s\times X_t=\gamma'\times\beta+t(\beta'\times\beta),
$$
we obtain
\begin{equation}\label{eq.luz-1}
 (X_s,X_t,V)=(\gamma',\beta,V)+t(\beta',\beta,V)    
\end{equation}	
and
\begin{equation}\label{eq.luz-2}
(X_s,X_t,X_{st})=(\gamma',\beta,\beta')+t(\beta',\beta,\beta')=(\gamma',\beta,\beta').
\end{equation}
Replacing \eqref{eq.luz-1} and \eqref{eq.luz-2} in \eqref{II}, we have
	\begin{equation}\label{42}
	(\gamma',\beta,V)(\gamma',\beta,\beta')+t(\beta',\beta,V)(\gamma',\beta,\beta')=-\interno{\gamma'}{\beta}^3
	\end{equation}	
which is equivalent to an identically zero polynomial, in the variable $t,$ with coefficients depending on $s,$ given by
 \begin{equation}\label{43}
p(t):=(\gamma',\beta,V)(\gamma',\beta,\beta')+\interno{\gamma'}{\beta}^3+t(\beta',\beta,V)(\gamma',\beta,\beta')=0.
	\end{equation}
It follows that 
\begin{equation}\label{zero-pol-light}
\begin{cases}
(\gamma',\beta,V)(\gamma',\beta,\beta')+\interno{\gamma'}{\beta}^3&=0\\
 (\beta',\beta,V)(\gamma',\beta,\beta')&=0.\\
\end{cases}
\end{equation}
 The second equation of \eqref{zero-pol-light} gives us two possibilities: $(\gamma',\beta,\beta')=0,$ which implies, from \eqref{42}, that $\gamma'$ is orthogonal to $\beta$, which leads us to a contradiction, since $F=\interno{\gamma'}{\beta}\neq0$, or $(\beta',\beta,V)=0$ which gives that $\{\beta',\beta,V\}$ is linearly dependent. In this last case, there exist smooth functions $a(s)$ and $b(s)$ such that 
	\begin{equation}\label{44}
	V=a(s)\beta(s)+b(s)\beta'(s).
	\end{equation}
Since $\beta$ is lightlike, by the item (iii) of Lemma \ref{basic-1}, p.\pageref{basic-1}, it follows that $\beta'$ is spacelike and we can choose a parameter $s$ such that $\interno{\beta'}{\beta'}=1.$ Thus, 
 $$
 \interno{V}{V}=a^2\interno{\beta}{\beta}+b^2\interno{\beta'}{\beta'}+2ab\interno{\beta}{\beta'}= b^2.
 $$
This gives that $b(s)=b$ is constant. We claim that $b\neq 0.$ Otherwise, if $b=0,$ we would have, from \eqref{44}, that $V$ would be parallel to $\beta$, which is an absurd, since from the first equation of \eqref{zero-pol-light}, we would have $F=\interno{\gamma'}{\beta}=0,$ which would imply that $M$ is degenerate. Thus, $b$ is a nonzero constant.
	 
Differentiating the equations $\interno{\beta'}{\beta'}=1$ and $\interno{\beta}{\beta'}=0$ we obtain, respectively, $\interno{\beta'}{\beta''}=0$ and $\interno{\beta}{\beta''}+1=0$. On the other hand, differentiating \eqref{44}, we have	\begin{equation}\label{eq.luz-3}
	    a'\beta+a\beta'+b\beta''=0.
	\end{equation}	
Taking the inner product of \eqref{eq.luz-3} with $\beta$ we obtain $b\interno{\beta}{\beta''}=0$, i.e., $\interno{\beta}{\beta''}=0,$ which is an absurd, since $\interno{\beta}{\beta''}=-1.$ Therefore, $\beta$ cannot be a lightlike vector field.
\end{proof}

The next result is the main result of this section, where we classify the non-cylindrical ruled surfaces that are translating solitons of the inverse mean curvature flow. In the previous lemma, we proved that $\beta$ cannot be a lightlike vector field in an open interval. Therefore, since our analysis is local and the property of being spacelike or timelike is an open property, we can assume that $\beta$ is not a lightlike vector field everywhere in $M.$

\begin{theorem}\label{theo-main}
Let $M$ be a non-cylindrical, nondegenerate, smooth ruled surface $\mathbb{L}^3$ parametrized by $X(s,t)=\gamma(s)+t\beta(s)$. If $M$ is a translating soliton of the inverse mean curvature flow, relative to the velocity vector $V=(v_1,v_2,v_3)\in\mathbb{L}^3$, then there exists a parameter $s$ such that 

\begin{itemize}
    \item[(i)] $\beta$ is a straight line parametrized by $\beta(s)=(1,s,s);$
    \item[(ii)] If $\gamma(s)=(x(s),y(s),z(s)),$ then $v_2\neq v_3$ and  
\begin{equation}\label{param-non-cyl}
\left\{\aligned
    x(s)&=\left(\dfrac{v_2-v_3}{8}\right)s + c_1\\
    y(s)&=\dfrac{3}{64}(v_2-v_3)s^2+\dfrac{v_1}{16}s + \left(\dfrac{v_3-3v_2}{32}\right)\ln|s| + c_2\\
    z(s)&=\frac{3}{64}(v_2-v_3)s^2+\frac{v_1}{16}s + \left(\dfrac{v_2-3v_3}{32}\right)\ln|s| + c_3,\\
\endaligned\right.
\end{equation}
where $c_1,c_2,$ and $c_3$ are constants (see Figure \ref{non-cyl}).

\end{itemize}
\end{theorem}
\begin{figure}[ht]
	\begin{center}
		\def\svgwidth{0.9\textwidth} 
		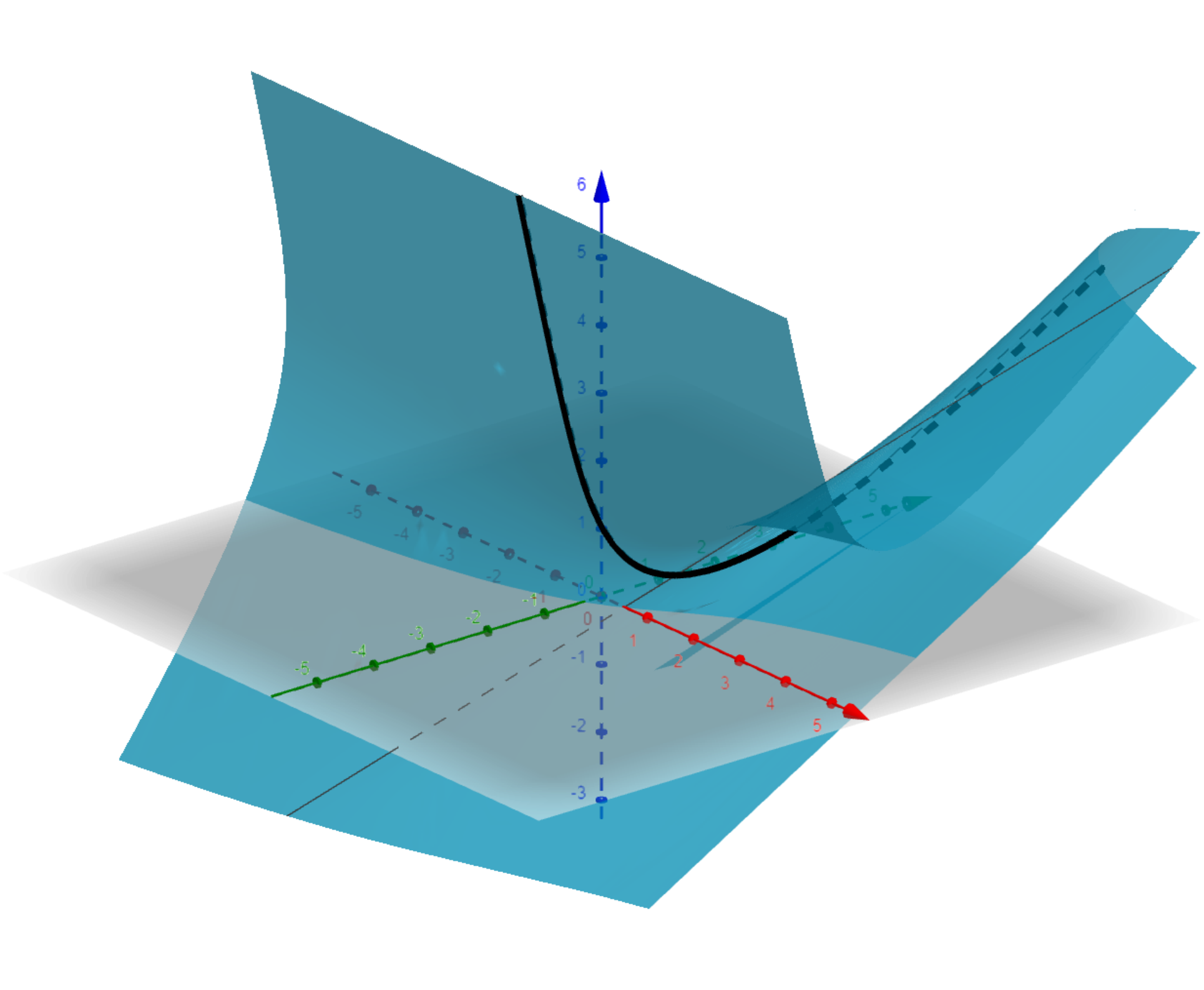
		\caption{A non-cylindrical translating soliton of the inverse mean curvature flow in $\mathbb{L}^3$ and its base curve. Here we are taking $V=(2,9,1).$ }\label{non-cyl}
	\end{center}
\end{figure}

\begin{proof}
Suppose that the parametrization $X(s,t)=\gamma(s)+t\beta(s)$ is orthogonal, i.e., $\interno{\beta}{\gamma'}=0.$ Since $\beta$ is not lightlike, by Lemma \ref{basic-1}, we can choose the parameter $s$ such that $\interno{\beta}{\beta}=:\delta\in\{-1,1\}$. In both cases, we have $\beta'$ orthogonal to $\beta.$ The coefficients of the first fundamental form are
\[
 \begin{cases} 
  E&=\interno{\gamma'}{\gamma'}+2\interno{\gamma'}{\beta'}t+\interno{\beta'}{\beta'}t^2,\\ F&=\interno{\beta}{\gamma'}=0,\\
  G&=\delta.\\
\end{cases}
\]
On the other hand,
$$
(X_s,X_t,X_{ss})=(\gamma',\beta,\gamma'')+[(\gamma',\beta,\beta'')+(\beta',\beta,\gamma'')]t+(\beta',\beta,\beta'')t^2.
$$	
Replacing these facts in \eqref{princ}, we obtain
\begin{equation} \label{45}	
   \begin{aligned}
   &[(\gamma',\beta,V)+t(\beta',\beta,V)]\times\\
   &\times [\delta\{(\gamma',\beta,\gamma'')+[(\gamma',\beta,\beta'')+(\beta',\beta,\gamma'')]t+(\beta',\beta,\beta'')t^2\}]\\
   &=2[\delta(\interno{\gamma'}{\gamma'}+2\interno{\gamma'}{\beta'}t+\interno{\beta'}{\beta'}t^2)]^2.
\end{aligned}
   \end{equation}
Making the calculations with the Equation \eqref{45}, we obtain the identically zero fourth degree polynomial $p(t)=\sum_{i=0}^{4}A_i(s)t^i,$ for
\begin{equation}\label{c-pol}
\left\{\begin{aligned}
    A_0&=2\interno{\gamma'}{\gamma'}^2-\delta(\gamma',\beta,V)(\gamma',\beta,\gamma''),\\
    A_1&=8\interno{\gamma'}{\gamma'}\interno{\gamma'}{\beta'}\\
    &\quad-\delta[(\gamma',\beta,V)(\beta',\beta,\gamma'')+(\gamma',\beta,V)(\gamma',\beta,\beta'')+(\beta',\beta,V)(\gamma',\beta,\gamma'')],\\
    A_2&=4\interno{\beta'}{\beta'}\interno{\gamma'}{\gamma'}+8\interno{\gamma'}{\beta'}^2\\
    &\quad-\delta[(\gamma',\beta,V)(\beta',\beta,\beta'')+(\beta',\beta,V)(\beta',\beta,\gamma'')+(\beta',\beta,V)(\gamma',\beta,\beta'')],\\
    A_3&=4\interno{\gamma'}{\beta'}\interno{\beta'}{\beta'}-\delta(\beta',\beta,V)(\beta',\beta,\beta''),\\
    A_4&=2\interno{\beta'}{\beta'}^2.\\
\end{aligned}\right.
\end{equation}
By using that $p(t)$ is an identically zero polynomial, we obtain that $A_4=0,$ i.e., $\interno{\beta'}{\beta'}=0.$ Since $\interno{\beta}{\beta}=\delta,$ $\interno{\beta}{\beta'}=0,$ and there is no lightlike vector orthogonal to timelike vectors in $\mathbb{L}^3$ (see Lemma \ref{basic-1}, item (iii)), we have that $\beta$ is spacelike and then $\delta=1.$ Thus $\beta'$ is a lightlike direction in the hyperboloid
$$
\{x\in\mathbb{L}^3:\interno{x}{x}=1\}.
$$
This implies that $\beta$ is a straight line. This implies that there exist vectors $\vec{a},\vec{b}\in\mathbb{L}^3$ such that $\beta(s)=\vec{a}s+\vec{b}$. Since $\beta'=\vec{a}$, we have 
\[
\interno{\vec{a}}{\vec{a}}=0=\langle\vec{a},\vec{b}\rangle\quad \mbox{and}\quad \langle\vec{b},\vec{b}\rangle=1.
\]
Therefore, applying rigid motions in $\mathbb{L}^3$ if necessary, we can consider $\vec{a}=(0,1,1)$ and $\vec{b}=(1,0,0)$. This implies that $\beta(s)=(1,s,s)$ and it proves item (i). 
				
Using item (i), we will consider the local parametrization of $M$ given by 
				
\[
X(s,t)=\gamma(s)+t(1,s,s).
\] 
Observe that $EG-F^2=\interno{\gamma'}{\gamma'}+2\interno{\vec{a}}{\gamma'}t$ and, replacing these facts in \eqref{c-pol}, we have $A_4=A_3=0$ and we obtain the system of equations
\begin{equation}\label{eq:A}
\left\{\begin{aligned}
  A_0&=2\interno{\gamma'}{\gamma'}^2-(\gamma',\beta,V)(\gamma',\beta,\gamma''),\\
  A_1&=8\interno{\gamma'}{\gamma'}\interno{\vec{a}}{\gamma'}-[(\gamma',\beta,V)(\vec{a},\beta,\gamma'')+(\vec{a},\beta,V)(\gamma',\beta,\gamma'')],\\
  A_2&=8\interno{\vec{a}}{\gamma'}^2-(\vec{a},\beta,V)(\vec{a},\beta,\gamma'').
\end{aligned}\right.
\end{equation}
We claim that $\interno{\vec{a}}{\gamma'}\neq 0$. Indeed, otherwise, considering \[\gamma(s)=(x(s),y(s),z(s)),\] if $y'-z'=\interno{\vec{a}}{\gamma'}=0$, it would follow, from the condition $\interno{\gamma'}{\beta}=0,$ that  
$$
\interno{\gamma'}{\beta}=\langle\vec{b},\gamma'\rangle+s\interno{\vec{a}}{\gamma'}=\langle\vec{b},\gamma'\rangle=x'=0,
$$
which would imply that 
$$
\interno{\gamma'}{\gamma'}=(x')^2+(y')^2-(z')^2=0,
$$
i.e. $\gamma$ is lightlike, giving $EG-F^2=0,$ which is an absurd since $M$ is nondegenerate. Thus, $\interno{\vec{a}}{\gamma'}\neq 0$. This gives $y'(s)\neq z'(s)$ for every $s.$
To conclude that $v_2-v_3\neq 0,$ we observe that 
$$
\vec{a}\times \beta=
\left|\begin{array}{rcr}
e_1 & e_2  & -e_3 \\ 
0 & 1 & 1\\
1 & s  & s
\end{array} \right|=(0,1,1)=\vec{a}.
$$ 
Replacing this fact in the expressions of $A_2=0$ given in \eqref{eq:A}, we obtain 
\begin{equation}\label{47}
8\interno{\vec{a}}{\gamma'}^2-\interno{\vec{a}}{V}\interno{\vec{a}}{\gamma''}=0.
\end{equation}
Since $\langle\vec{a},\gamma'\rangle\neq 0,$ it follows that $\interno{\vec{a}}{V}\neq 0.$ This gives $v_2\neq v_3.$ On the other hand, Equation \eqref{47} is equivalent to
$$
(v_2-v_3)(y''(s)-z''(s))-8(y'(s)-z'(s))^2=0.
$$
Defining $u(s)=y'(s)-z'(s)$ we obtain the separable ODE
$$
(v_2-v_3)u'(s)-8(u(s))^2=0,
$$
whose solution (after a translation of the parameter $s$) is		
\begin{equation}\label{y-z}
y'(s)-z'(s)=u(s)=\frac{v_3-v_2}{8s}.
\end{equation}
On the other hand, since $\interno{\gamma'}{\beta}=0$ and \eqref{y-z} imply
\begin{equation}\label{exp.xp}
x'(s)=-(y'(s)-z'(s))s=-\frac{v_3-v_2}{8},
\end{equation}
the expression for $x(s)$ follows by integrating \eqref{exp.xp} relative to the variable $s.$ 

From the expression of $A_0=0$ and $A_2=0$ given in \eqref{eq:A}, we obtain
\begin{equation}\label{A0A2}(\gamma',\beta,V)(\gamma',\beta,\gamma'')=2\langle \gamma',\gamma'\rangle^2 \quad \mbox{and} \quad (\vec{a},\beta,V)(\vec{a},\beta,\gamma'')=8\langle\vec{a},\gamma'\rangle^2.    
\end{equation}
On the other hand, multiplying the expression of $A_1=0,$ given in \eqref{eq:A}, by $\frac{1}{2}(\vec{a},\beta,V)\cdot(\gamma',\beta,V)$  and replacing the resulting expressions in \eqref{A0A2}, we have
\[
\aligned
4(\gamma',\beta,V)^2\interno{\vec{a}}{\gamma'}^2&+(\vec{a},\beta,V)^2\interno{\gamma'}{\gamma'}^2\\
&-4\interno{\gamma'}{\gamma'}\interno{\vec{a}}{\gamma'}(\gamma',\beta,V)(\vec{a},\beta,V)=0,
\endaligned
\]
i.e.,
\[
[2(\gamma',\beta,V)\interno{\vec{a}}{\gamma'}-(\vec{a},\beta,V)\interno{\gamma'}{\gamma'}]^2=0,
\]
which gives
\begin{equation}\label{eq:alg-2}
2(\gamma',\beta,V)\interno{\vec{a}}{\gamma'}=(\vec{a},\beta,V)\interno{\gamma'}{\gamma'}.
\end{equation}
On the other hand, since
\[
s\interno{\vec{a}}{\gamma'}=s(y'-z')=\frac{v_3-v_2}{8},\quad
(\vec{a},\beta,V)=\interno{\vec{a}}{V}=v_2-v_3,
\]
\[
\aligned
s(\gamma',\beta,V)&=s\left|
\begin{array}{ccc}
    x'&y'&z'\\
    1&s&s\\
    v_1&v_2&v_3\\
\end{array}
\right|\\
&=v_1s^2(y'-z')-v_2s(sx'-z')+v_3s(sx'-y')\\
&=-v_1sx' + s^2x'(v_3-v_2)+v_2(x'+sy')-v_3sy'\\
&=x'(-v_1s+(v_3-v_2)s^2+v_2)+s(v_2-v_3)y'\\
&=-\left(\frac{v_3-v_2}{8}\right)((v_3-v_2)s^2-v_1s+v_2) + s(v_2-v_3)y',
\endaligned
\]
and
\[
\aligned
s^2\interno{\gamma'}{\gamma'}&=s^2(x')^2+s^2(y')^2-s^2(z')^2\\
&=s^2(x')^2+s^2(y')^2-(x'+sy')^2\\
&=(x')^2(s^2-1)-2sx'y'\\
&=\left(\frac{v_3-v_2}{8}\right)^2(s^2-1)+2s\left(\frac{v_3-v_2}{8}\right)y',
\endaligned
\]
we have, after multiplying \eqref{eq:alg-2} by $s^2$, that
\[
\aligned
-2\left(\frac{v_3-v_2}{8}\right)^2&\left((v_3-v_2)s^2-v_1s+v_2\right)-2s(v_3-v_2)\left(\frac{v_3-v_2}{8}\right)y'\\
&=(v_3-v_2)\left(\frac{v_3-v_2}{8}\right)^2(s^2-1)\\
&\quad+2s(v_3-v_2)\left(\frac{v_3-v_2}{8}\right)y'.
\endaligned
\]
This gives, after simplifications,
\[
\aligned
4s(v_3-v_2)y'(s)&= -\left(\frac{v_3-v_2}{8}\right)\times\\
&\quad\times\left(2(v_3-v_2)s^2-2v_1s+2v_2 + (v_3-v_2)(s^2-1)\right),
\endaligned
\]
i.e.,
\begin{equation}\label{y-prime-nc}
\aligned
y'(s)&=-\frac{3(v_3-v_2)s^2-2v_1s+3v_2-v_3}{32s}\\
&=-\frac{3(v_3-v_2)}{32}s + \frac{v_1}{16}-\frac{3v_2-v_3}{32s}.
\endaligned
\end{equation}
The expression of $y(s)$ comes after integration in the variable $s.$ The expression of $z(s)$ comes after replacing \eqref{y-prime-nc} into \eqref{y-z} and then integrating in $s.$
\end{proof}

	\begin{corollary}
There is no translating soliton of the inverse mean curvature flow in $\mathbb{L}^3$ that is a conical ruled surface.
	\end{corollary}
	\begin{proof}
It follows from item (ii) of Theorem \ref{theo-main} that $\interno{\vec{a}}{V}=v_2-v_3\neq0$, where $V=(v_1,v_2,v_3)\in\mathbb{L}^3$. By \eqref{exp.xp}, p.\pageref{exp.xp}, the first coordinate $x(s),$ of the parametrization of $\gamma$, is not constant. Thus $\gamma$ does not reduce to a single point, i.e., the surface cannot be conical.
	\end{proof}

\section{Cylindrical translating solitons}\label{cyl}
In this section, we classify the cylindrical nondegenerate ruled surfaces that are translating solitons of the IMCF in $\mathbb{L}^3.$ These surfaces have the parametrization
\[
X(s,t)=\gamma(s)+tw
\]
where $\gamma$ is a curve parametrized by the arc length and $w$ is a constant vector. Up to rigid motions of $\mathbb{L}^3,$ we can choose
\[
w=(1,0,1),\ w=(1,0,0), \ \mbox{or} \ w=(0,0,1).
\]
With this parametrization, we have $X_t=\gamma'(s)$ and $X_t=w,$ which implies $E=\interno{\gamma'}{\gamma'}=\delta\in\{-1,1\}$, $F=\interno{\gamma'}{w}$ and $G=\interno{w}{w}$. The case $w=(1,0,1)$ cannot happen for the IMCF since $w$ being lightlike and $X_{st}=X_{tt}=0$ imply 
\begin{equation}
        H=-\frac{1}{2}\frac{eG}{EG-F^2}=-\frac{1}{2}\frac{e\interno{w}{w}}{\interno{w}{w}-\interno{\gamma'}{w}^2}=0.
    \end{equation}
Thus we have that $w=(1,0,0)$ (i.e., $w$ is spacelike) or $w=(0,0,1)$ (i.e., $w$ is timelike). We start analyzing the case when $w=(1,0,0).$

Let $M$ be a nondegenerate, cylindrical, smooth ruled surface $\mathbb{L}^3$ parametrized by 
\[
X(s,t)=\gamma(s)+t(1,0,0),
\]
where $\gamma$ is a curve parametrized by the arc length, lying in a plane orthogonal to $w$, i.e., 
$$
\gamma(s)=(0,x(s),y(s))\quad\mbox{and}\quad(x'(s))^2-(y'(s))^2=\delta\in\{-1,1\}.
$$   
Under this parametrization, we obtain $X_s=\gamma'(s)=(0,x'(s),y'(s))$ and $X_t=(1,0,0),$ which implies $E=\interno{\gamma'}{\gamma'}=\delta$, $F=0$ and $G=1$. Noticing that $X_s\times X_t=(0,y'(s),x'(s))$ and $(EG-F^2)^2=1$, the translating soliton equation \eqref{princ}, for a velocity vector $V=(v_1,v_2,v_3),$ becomes
		\begin{equation}\label{trans-cyl}
  (v_2y'(s)-v_3x'(s))(y'(s)x''(s)-x'(s)y''(s))=2.		    
		\end{equation}
Since $(x'(s))^2-(y'(s))^2=\delta$, it follows, after differentiating in $s,$ that $x'(s)x''(s)-y'(s)y''(s)=0.$ Thus, we have the linear system of equations, in the unknowns  $x''(s)$ and $y''(s),$
		$$\begin{cases}
		x'(s)x''(s)-y'(s)y''(s)=0\\
		y'(s)x''(s)-x'(s)y''(s)=\displaystyle\frac{2}{v_2y'(s)-v_3x'(s)},
		\end{cases}
		$$
whose solution is
\begin{equation}\label{eq:xypp}
  x''(s)=\frac{-2\delta y'(s)}{v_2y'(s)-v_3x'(s)}\quad\text{and}\quad y''(s)=\frac{-2 \delta x'(s)}{v_2y'(s)-v_3x'(s)}.
\end{equation}
 By multiplying the first equation in \eqref{eq:xypp} by $-v_2,$ the second equation in \eqref{eq:xypp} by $v_3$ and summing the results, we obtain 
\begin{equation}\label{eq.delta}
  v_3y''(s)-v_2x''(s)=2\delta.        
\end{equation}
Integrating \eqref{eq.delta} in the variable $s,$ and making, if necessary, a change of variables by translation of $s$ to avoid the constant after integration, we obtain
	\begin{equation}\label{51}
		 v_3y'(s)-v_2x'(s)=2\delta s.
		\end{equation}
Thus, we have the following linear system, in the unknowns $x'$ and $y'$
		\begin{equation}
			\begin{cases}\label{52}
			v_3y'(s)-v_2x'(s)=2\delta s\\
			(x'(s))^2-(y'(s))^2=\delta.
			\end{cases}
		\end{equation}
By multiplying the second equation of \eqref{52} by $v_3^2$ and replacing the first equation in the resulting equation, we obtain
\begin{equation}\label{Eq-cyl-space}
    (v_3^2-v_2^2)(x'(s))^2 -4\delta v_2 s x'(s) - (4s^2+\delta v_3^2)=0.
\end{equation}
Clearly, the analysis of \eqref{Eq-cyl-space} must be divided in the cases when $v_2^2\neq v_3^2$ and $v_2^2=v_3^2.$ Notice that, since $v_2$ and $v_3$ cannot be simultaneously zero, the second case implies that $v_2\neq0$ and $v_3\neq 0.$ For the sake of clarity, we present these two cases in two different theorems, starting with the case when $v_2^2=v_3^2.$
  
\begin{theorem}\label{theo-cyl-1}
Let $M$ be a nondegenerate, smooth, cylindrical ruled surface in $\mathbb{L}^3,$ parametrized by  $X(s,t)=\gamma(s)+t w,$ where $w=(1,0,0)$. If $M$ is a translating soliton of the inverse mean curvature flow, relative to the velocity vector $V=(v_1,v_2,\pm v_2),$ with $v_2\neq0,$ then the base curve $\gamma(s)=(0,x(s),y(s))$ has the parametrization
\begin{equation}\label{eq:teo6}
    \left\{\begin{aligned}
  x(s)&=-\dfrac{ \delta s^2}{2v_2}-\frac {v_2}{4}\ln\vert s\vert +c_1\\
  y(s)&=\pm\left(\frac{\delta s^2}{2v_2}-\frac{v_2}{4}\ln\vert s\vert \right)+c_2,
    \end{aligned}\right.
\end{equation}
where $c_1,c_2\in\mathbb{R}$ are constants (see Figure \ref{trans-cyl-case-1}).
\end{theorem}
\begin{figure}[ht]
	\begin{center}
		\def\svgwidth{1\textwidth} 
		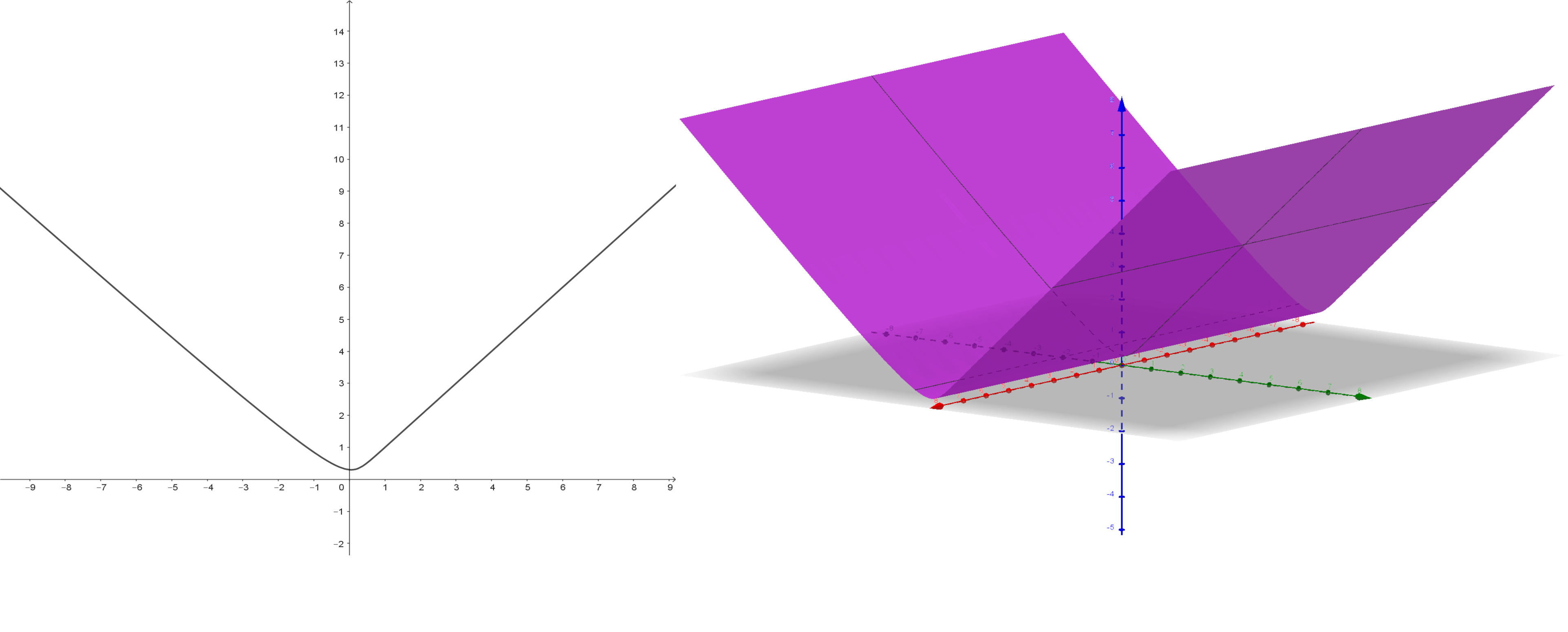
		\caption{Example of a cylindrical translating soliton of the IMCF given in Theorem \ref{theo-cyl-1}. Here we are taking $v_2=v_3=1$ and $\delta=1$}\label{trans-cyl-case-1}
	\end{center}
\end{figure}

\begin{proof}
If the velocity vector $V=(v_1,v_2,v_3)$ of the translating soliton is such that $v_2=\pm v_3,$ it follows, from Equation \eqref{Eq-cyl-space}, that 
 \begin{equation}\label{x-prime}
 \begin{aligned}
      x'(s)&=-\frac{\delta s}{v_2}-\frac{v_3^2}{4v_2}\frac{1}{s}\\
     &=-\frac{\delta s}{v_2}-\frac{ v_2}{4}\frac{1}{s}.
\end{aligned}
 \end{equation}
Integrating \eqref{x-prime} by $x'(s),$ we obtain
 $$x(s)=-\frac{\delta s^2}{2v_2}-\frac{ v_2}{4}\ln\vert s\vert +c_1.$$
On the other hand, by replacing \eqref{x-prime} into \eqref{51} we have 
 \begin{equation}\label{y-prime}
 \begin{aligned}
v_3y'(s)&=v_2x'(s)+2\delta s\\
&=v_2\left(-\frac{ s}{v_2}-\frac{v_2}{4s}\right)+2\delta s\\
&= \delta s - \frac{ v_2^2}{4s}.
\end{aligned}
 \end{equation}
Integrating \eqref{y-prime} we obtain
 \[
v_3y(s)=\frac{ \delta s^2}{2} - \frac{ v_2^2}{4}\ln|s| + k,
 \]
 i.e.,
\[
y(s)=\pm\left(\frac{ \delta s^2}{2v_2} - \frac{ v_2}{4}\ln|s|\right) + c_2,
\] 
where $v_3=\pm v_2,$ $c_1,k$ and $c_2=k/v_3$ are constants.
\end{proof}	

\begin{theorem}\label{theo-cyl-2}
Let $M$ be a nondegenerate, smooth, cylindrical, ruled surface in $\mathbb{L}^3,$ parametrized by  $X(s,t)=\gamma(s)+t w,$ where $w=(1,0,0)$. If $M$ is a translating soliton of the inverse mean curvature flow, relative to the velocity vector $V=(v_1,v_2,v_3),$ with $v_2^2\neq v_3^2$, then the base curve $\gamma(s)=(0,x(s),y(s))$ has the parametrization (see Figure \ref{trans-cyl-case-2}) 
\begin{equation}\label{para-cyl-gen}
\left\{\begin{aligned}
x(s)&=\frac{ \delta v_2}{v_3^2-v_2^2}s^2\\
&\quad\pm\frac{|v_3|}{v_3^2-v_2^2}\left[s\sqrt{s^2+\frac{\delta(v_3^2-v_2^2)}{4}} + \ln\left|s+ \sqrt{s^2+\frac{\delta(v_3^2-v_2^2)}{4}}\right|\right]+c_1\\
y(s)&=\frac{\delta v_3}{v_3^2-v_2^2}s^2\\
&\quad\pm\frac{|v_2|}{v_3^2-v_2^2}\left[s\sqrt{s^2+\frac{\delta(v_3^2-v_2^2)}{4}} + \ln\left|s+ \sqrt{s^2+\frac{\delta(v_3^2-v_2^2)}{4}}\right|\right]+c_2.
\end{aligned}\right.
\end{equation}
\end{theorem}
\begin{figure}[h]
	\begin{center}
		\def\svgwidth{1\textwidth} 
		\input{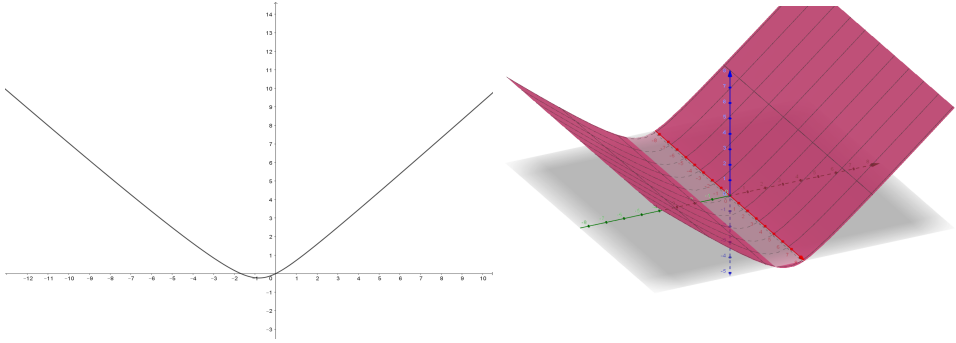}
		\caption{Example of a cylindrical translating soliton of the IMCF given in Theorem \ref{theo-cyl-2}. Here we are taking $v_2=1,$ $v_3=2$ and $\delta=1.$}\label{trans-cyl-case-2}
	\end{center}
\end{figure}
\begin{proof}
 If $v_2^2\neq v_3^2$ we obtain, from \eqref{Eq-cyl-space}, that
 $$x'(s)=\frac{4\delta v_2 s\pm\sqrt{16v_2^2 s^2+4(v_3^2-v_2^2)(4s^2+v_3^2\delta)}}{2(v_3^2-v_2^2)}$$
or, equivalently,
 \begin{equation}\label{54}
   x'(s)=\frac{2\delta v_2 s}{v_3^2-v_2^2}\pm \frac{2|v_3|}{v_3^2-v_2^2}\sqrt{s^2+\frac{\delta(v_3^2-v_2^2)}{4}}.
 \end{equation}
The expression of $x(s)$ is obtained after integration of \eqref{54} in the variable $s$ (analysing the cases $\delta(v_3^2-v_2^2)>0$ and $\delta(v_3^2-v_2^2)<0$ separately). The expression of $y(s)$ comes by replacing the expression of $x'(s)$ in the system \eqref{52} and integrating the result in $s.$
\end{proof}

Now let us consider the case when $w=(0,0,1)$ and $M$ has the parametrization
\[
X(s,t)=\gamma(s)+t(0,0,1).
\]
\begin{theorem}\label{theo-cyl-3}
Let $M$ be a nondegenerate, smooth, cylindrical, ruled surface in $\mathbb{L}^3,$ parametrized by  $X(s,t)=\gamma(s)+t w,$ where $w=(0,0,1)$. If $M$ is a translating soliton of the inverse mean curvature flow, relative to the velocity vector $V=(v_1,v_2,v_3),$ then the base curve $\gamma(s)=(x(s),y(s),0)$ has the parametrization (see Figure \ref{trans-cyl-case-3})
\begin{equation}\label{para-cyl-gen-time}
\left\{\begin{aligned}
x(s)&=-\frac{v_1}{v_1^2+v_2^2}s^2\mp\frac{|v_2|}{4}\arccos\left(\frac{2s}{\sqrt{v_1^2+v_2^2}}\right)\\
&\quad\pm \frac{|v_2|s}{v_1^2+v_2^2}\sqrt{\frac{v_1^2+v_2^2}{4}-s^2}+c_1\\
y(s)&=-\frac{v_2}{v_1^2+v_2^2}s^2\pm\frac{|v_1|}{4}\arccos\left(\frac{2s}{\sqrt{v_1^2+v_2^2}}\right)\\
&\quad\mp \frac{|v_1|s}{v_1^2+v_2^2}\sqrt{\frac{v_1^2+v_2^2}{4}-s^2}+c_2.
\end{aligned}\right.
\end{equation}
\end{theorem}

\begin{figure}[ht]
	\begin{center}
		\def\svgwidth{0.8\textwidth} 
		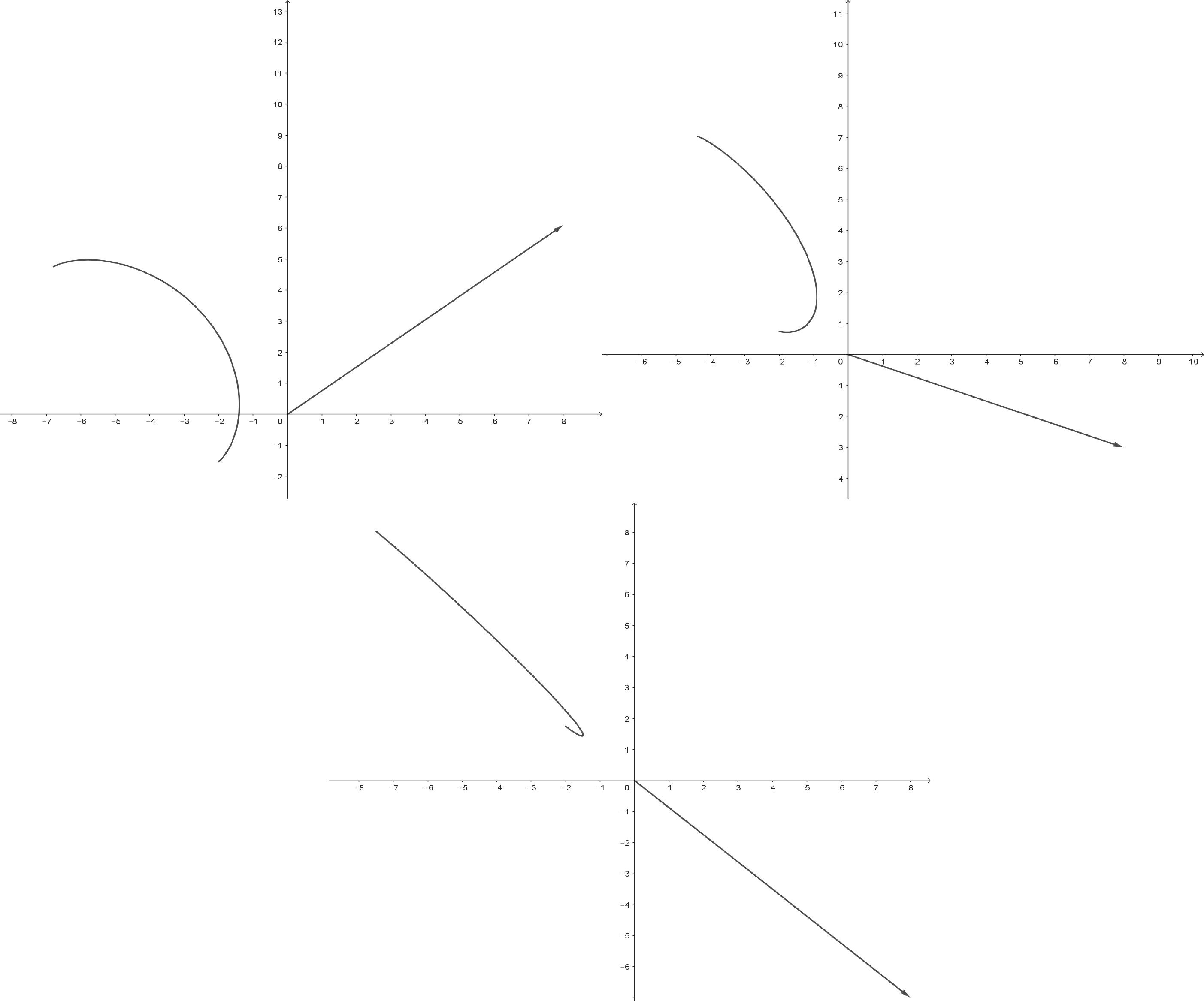
		\caption{Examples of ``cycloidal'' base curves for cylindrical translating soliton of the IMCF given in Theorem \ref{theo-cyl-3}. Here we are taking, respectively, $(v_1,v_2)=(8,6),$ $(v_1,v_2)=(8,-3),$ and $(v_1,v_2)=(8,-7).$}\label{trans-cyl-case-3}
	\end{center}
\end{figure}
\begin{proof}
Taking again the base curve $\gamma$ in a plane orthogonal to $w$ and parametrizing $\gamma$ by the arc length, we have
\[
\gamma(s)=(x(s),y(s),0), \quad \mbox{and} \quad (x'(s))^2+(y'(s))^2=1.
\]
Since $X_s=(x'(s),y'(s),0)$ and $X_t=(0,0,1),$ we obtain $E=1,$ $F=0,$ and $G=-1.$ Thus, the translating soliton equation \eqref{princ} becomes
  \begin{equation}\label{w-tempo}
  (v_1y'(s)-v_2x'(s))(y'(s)x''(s)-x'(s)y''(s))=-2.
  \end{equation}
Differentiating the arc length equation and using \eqref{w-tempo}, we obtain the system
\[
\begin{cases}
    x'(s)x''(s)+y'(s)y''(s)&=0\\
    y'(s)x''(s)-x'(s)y''(s)&=\dfrac{-2}{v_1y'(s)-v_2x'(s)},
\end{cases}
\]
whose solution is
\[
x''(s)=\frac{-2y'(s)}{v_1y'(s)-v_2x'(s)} \quad \mbox{and} \quad y''(s)=\frac{2x'(s)}{v_1y'(s)-v_2x'(s)}.
\]
This gives
\[
v_2y''(s)+v_1x''(s)=-2,
\]
i.e. (after a translation of the parameter $s$),
\[
v_2y'(s)+v_1x'(s)=-2s.
\]
This equation, together with the arc length equation, gives us the system of equations
\begin{equation}\label{sys-time}
    \begin{cases}
        v_2y'(s)+v_1x'(s)&=-2s\\
    (x'(s))^2+(y'(s))^2&=1.
    \end{cases}
\end{equation}
By multiplying the second equation of \eqref{sys-time} by $v_2^2$ and then replacing the first equation of this system in the resulting equation, in order to eliminate $y'(s),$ we obtain
\begin{equation}\label{eq-cyl-time}
    (v_1^2+v_2^2)(x'(s))^2+4v_1sx'(s)+4s^2-v_2^2=0,
\end{equation}
which gives
\begin{equation}\label{x-prime-cyl-time}
 \aligned
 x'(s)&=\dfrac{-4v_1s\pm\sqrt{16v_1^2s^2-4(v_1^2+v_2^2)(4s^2-v_2^2)}}{2(v_1^2+v_2^2)}\\
 &=-\frac{2v_1}{v_1^2+v_2^2}s \pm \frac{2|v_2|}{v_1^2+v_2^2}\sqrt{\frac{v_1^2+v_2^2}{4}-s^2}.
\endaligned
\end{equation}
Replacing the expression of $x'(s)$ given in \eqref{x-prime-cyl-time} in the system of equations \eqref{sys-time}, we obtain
\begin{equation}\label{y-prime-cyl-time}
    y'(s)=-\frac{2v_2}{v_1^2+v_2^2}s \mp \frac{2|v_1|}{v_1^2+v_2^2}\sqrt{\frac{v_1^2+v_2^2}{4}-s^2}.
\end{equation}
Integrating \eqref{x-prime-cyl-time} and \eqref{y-prime-cyl-time} in the variable $s,$ we obtain the result.
\end{proof}

\end{document}